\documentclass[runningheads,envcountsame]{llncs}
\usepackage{graphicx}
\usepackage[english]{babel}
\usepackage{amssymb,amsmath}
\usepackage{amscd,amsfonts}
\usepackage{mathtools}
\usepackage{array}
\usepackage{hyperref}
\usepackage[shortlabels]{enumitem}
\usepackage{algorithm,algpseudocode}
\usepackage[all,2cell,ps]{xy}
\usepackage{multicol}

\algnewcommand\algorithmicforeach{\textbf{for each}}
\algdef{S}[FOR]{ForEach}[1]{\algorithmicforeach\ #1\ \algorithmicdo}

\usepackage{xspace}\xspaceaddexceptions{\,}

\newcommand{\Var}{\mathnormal{V\mkern-.8\thinmuskip ar}}

\newcommand\A{{\mathbf A}}
\newcommand\B{{\mathbf B}}

\newcommand\IBSL{\ensuremath{\mathcal{IBSL}}\xspace}
\newcommand\LIBSL{\ensuremath{\mathcal{L\text{-}IBSL}}\xspace}

\newcommand\CL{\ensuremath{\mathrm{CL}}\xspace}

\newcommand\PWK{\ensuremath{\mathrm{PWK}}\xspace}

\DeclareMathAlphabet{\mathbfsf}{\encodingdefault}{\sfdefault}{bx}n
\makeatletter
\providecommand*{\Dashv}{\mathrel{\mathpalette\@Dashv\vDash}}
\newcommand*{\@Dashv}[2]{\reflectbox{$\m@th#1#2$}}
\makeatother

\renewcommand\geq{\geqslant}

\newcommand\leqs{\leqslant}

\usepackage{units}

\newcommand\pair[1]{{\langle#1\rangle}}

\newcommand\PL{{\mathcal{P}_{l}}}
\newcommand\PLA{{\mathcal{P}_{l} (\mathbb{A})}}

\useshorthands{"}
\defineshorthand{"-}{{}\nobreakdash-\hspace{0pt}}	

\newcommand{\?}{\ensuremath{\mkern0.4\thinmuskip}}   

\begin{document}

\title{Counting finite linearly ordered involutive bisemilattices}

\author{Stefano Bonzio\inst{1} \and
Michele Pra Baldi\inst{2} \and
Diego Valota\inst{3}}

\authorrunning{S. Bonzio et al.}

%
\institute{Dipartimento di Scienze Biomediche e Sanit\`a Pubblica,
	Universit\`{a} Politecnica delle Marche, \\
	 Via Tronto 10/a, 60200 Torrette di Ancona, Italy \\
	\email{stefano.bonzio@gmail.com}
\and
	Dipartimento  F.I.S.P.P.A., Universit\`{a} di Padova, Padova, Italy \\
	\email{m.prabaldi@gmail.com}
\and
	Dipartimento di Informatica, Universit\`{a} degli Studi di Milano,\\
	Via Comelico 39, I-20135 Milano, Italy	\\
	\email{valota@di.unimi.it}}

\maketitle

\begin{abstract}

The class of involutive bisemilattices plays the role of the algebraic counterpart
of paraconsistent weak Kleene logic. 
Involutive bisemilattices can be represented as \emph{P\l onka sums} of Boolean algebras,
that is semilattice direct systems of Boolean algebras.
In this paper we exploit the P\l onka sum representation with the aim of counting, up to isomorphism, finite involutive bisemilattices 
whose direct system is given by totally ordered semilattices.

\keywords{Finite Involutive Bisemilattices \and Weak Kleene Logic \and P\l onka sums.}
\end{abstract}

\section{Introduction}

The class of involutive bisemilattices plays the role of the algebraic counterpart among one of the three-valued logics introduced by Kleene in \cite{Kleene}, namely paraconsistent weak Kleene logic -- \PWK for short.  \PWK, essentially introduced by Halld\'en \cite{Hallden}, can be defined as the logic induced by a matrix given by the weak Kleene tables with $\{1,n\}$ as truth set: 

\begin{center}\renewcommand{\arraystretch}{1.2}
\begin{tabular}{>{$}c<{$}|>{$}c<{$}>{$}c<{$}>{$}c<{$}}
   \land & 0 & n & 1 \\[.2ex]
 \hline
       0 & 0 & n & 0 \\
       n & n & n & n \\          
       1 & 0 & n & 1
\end{tabular}
\qquad
\begin{tabular}{>{$}c<{$}|>{$}c<{$}>{$}c<{$}>{$}c<{$}}
   \lor & 0 & n & 1 \\[.2ex]
 \hline
     0 & 0 & n & 1 \\
     n & n & n & n \\          
     1 & 1 & n & 1
\end{tabular}
\qquad
\begin{tabular}{>{$}c<{$}|>{$}c<{$}}
  \lnot &  \\[.2ex]
\hline
  1 & 0 \\
  n & n \\
  0 & 1 \\
\end{tabular}

\end{center}

Equivalently (see \cite{CiuniCarrara,Bonzio16}), \PWK can be obtained out of (propositional) classical logic (\CL) imposing the following syntactical restriction: 
\[
\Gamma \vdash_{\PWK} \varphi \Longleftrightarrow \text{ there is }\Delta \subseteq \Gamma \text{ s.t. }\Var(\Delta)\subseteq\Var(\varphi) \text{ and } \Delta\vdash_{\CL}\varphi,
\]
where $\Var(\varphi)$ is the set of variables really occurring in $\varphi$.

Involutive bisemilattices consist of a \emph{regular} variety, namely one satisfying identities of the form $\varepsilon\approx\tau$, where $\Var(\varepsilon) = \Var(\tau)$. More precisely, involutive bisemilattices satisfy only the regular identities holding in Boolean algebras. Due to the general theory of regular varieties, which traces back to the pioneering work of P\l onka \cite{Plo67}, involutive bisemilattices can be represented as \emph{P\l onka sums} of Boolean algebras, that is, a sum over semilattice direct systems of Boolean algebras.
Over the years, P\l onka sums and (some) regular varieties have been studied in depth both from a purely algebraic perspective \cite{Bal70,Kal71,Harding2016,Harding20172} and in connection with their topological duals \cite{Romanowska96,Romanowska97,SB18}. The machinery of P\l onka sums has also found useful applications in the study of the constraint satisfaction problem \cite{Bergman2015} and in database semantics \cite{Libkin,Puhlmann}. Recently, thanks to the extension of this formalism to logical matrices \cite{BonzioMorascoPrabaldi,BonzioPraBaldi}, P\l onka sums have turned out to play a useful role in the investigation of logics featuring the presence of a non-sensical, infectious truth-value. This family of logics -- including \PWK and Bochvar logic \cite{Bochvar} -- provides valuable formal instruments to model computer-programs affected by errors \cite{Ferguson}. 
In this paper we exploit the P\l onka sum representation for the purpose of counting the finite members of a specific subclass of involutive bisemilattices, whose representation consists of a linearly ordered semilattice. In particular, we provide an algorithm offering a solution to the fine spectrum problem 
\cite{T75} for the class of linearly ordered involutive bisemilattices. In order to achieve this goal, we use the categorical apparatus developed in \cite{Loi}. 
We believe that the application of the above-mentioned algebraic
methods allows us to develop algorithms that are more efficient than
``brute-force'' procedures. This is confirmed by the computational
experiments. In particular, a comparison between the efficiency
of the algorithm introduced in this paper and of Mace4 is briefly
discussed in Section \ref{sec: Conclusioni}.

\section{Preliminaries}
A \textit{semilattice} is an algebra $\A = \langle A, \lor\rangle$, where $\lor$ is a binary commutative, associative and idempotent operation. Given a semilattice $\A$ and $a, b \in A$, we set $a \leq b \Longleftrightarrow a \lor b = b$.
It is easy to see that $\leq$ is a partial order on $A$.

We briefly recall the category of semilattice direct systems, introduced in \cite{Loi,SB18}. Intuitively, they consists of a specialization of direct (and inverse) systems of an arbitrary category, obtained by assuming the index set to be a semilattice instead of a (directed) pre-ordered set. For any unexplained notion in category theory, the reader is referred to \cite{ML98}. 

\begin{definition}\label{def: direct system}
 Let $ \mathfrak{C} $ be an arbitrary category. A \emph{semilattice direct system} in $ \mathfrak{C} $ is a triple $\mathbb{X}= \pair{X_{i}, p_{ii'}, I} $ such that 
 \begin{enumerate}
\item $ I $ is a semilattice.
\item $ \{X_i\}_{i\in I} $ forms an indexed family of objects in $ \mathfrak{C}$ with disjoint universes;
\item $p_{ii'} : X_{i} \to X_{i'}$ is a morphism of $ \mathfrak{C} $, for each pair $i\leqs i'$ ($i,i^{\prime}\in I$), satisfying that $p_{ii} $ is the identity in $ X_i $ and such that $ i\leq i'\leq i'' $ implies $ p_{i'i''} \circ p_{ii'}  = p_{ii''}$. 
\end{enumerate}
\end{definition}

We refer to $ I $, $X_i$ and $ p_{ii'} $ as the index set, the terms and the transition morphisms, respectively, of the (semilattice direct) system.

A morphism between two semilattice direct systems $ \mathbb{X}=\pair{X_{i}, p_{ii'}, I} $ and $ \mathbb{Y}=\pair{Y_{j}, p_{jj'}, J} $ is a pair $(\varphi, \{f_i\}_{i\in I})\colon\mathbb{X}\to\mathbb{Y}$ such that
\begin{itemize}
\item[i)] $ \varphi\colon I\rightarrow J $ is a semilattice homomorphism 
\item[ii)] $ f_i\colon X_{i}\rightarrow Y_{\varphi(i)} $ is a morphism in $ \mathfrak{C} $,
 making the following diagram
commutative, for each $ i, i'\in I $, $ i\leq i' $.
\end{itemize}
\begin{center}\leavevmode
\xymatrix{
X_i\ar@{->}[r]^-{p_{ii'}}\ar@{->}[d]_-{f_{i}}&X_{i'}\ar@{->}[d]^-{f_{i'}}\\
Y_{\varphi(i)} \ar@{->}[r]^-{q_{\varphi{i}\varphi{i'}}} 	& Y_{\varphi(i')}
	}
\end{center}
It is easy to check that the semilattice direct systems of a category $\mathfrak{C}$ form a category, which we denote by Sem-dir-$\mathfrak{C}$. Semilattice \emph{inverse} systems of an arbitrary category are obtained analogously, by, intuitively, reversing the directions of transition morphisms (see \cite{Loi} for precise details). Moreover, provided that two categories $\mathfrak{C}$ and $\mathfrak{D}$ are dually equivalent, the duality can be lifted to Sem-dir-$\mathfrak{C}$ and Sem-inv-$\mathfrak{D}$ (see \cite[Remark 3.6]{Loi}).
The class of \emph{involutive bisemilattices} has been introduced in \cite{Bonzio16} as the most suitable candidate to be the algebraic counterpart of the logic \PWK. 

\begin{definition}\label{def: IBSL}
\normalfont
An \emph{involutive bisemilattice} is an algebra $\B = \pair{B, \vee,\wedge,\neg,0,1}$ of type $(2,2,1,0,0)$ satisfying:
\begin{multicols}{2}
\begin{enumerate}[label=\textbf{I\arabic*}.]
\item $x \vee x\approx x$;
\item $x \vee y\approx y\vee x$;
\item $x\vee(y\vee z)\approx(x\vee y)\vee z$;
\item $\neg(\neg x)\approx x$;
\item $x\wedge y\approx\neg( \neg x\vee \neg y)$;
\item $x\wedge( \neg x\vee y)\approx x\wedge y$; \label{rmp}
\item $0\vee x\approx x$;
\item $1\approx \neg 0$.
\end{enumerate}
\end{multicols}
\end{definition}

Involutive bisemilattices form an equational class denoted by $\mathcal{IBSL}$. Examples of involutive bisemilattices include any Boolean algebra, as well as any semilattice with zero. In the latter case, the two binary operations coincide and the unary operation is the identity.
The variety $\mathcal{IBSL}$ is the \emph{regularization}\footnote{For the theory of regular varieties and regularizations we refer the reader to \cite{Romanowska92}.} of the variety $\mathcal{BA}$, of Boolean algebras (see \cite{plonka1984sum,Bonzio16}), i.e. $\mathcal{IBSL}\models\varepsilon\approx\tau$ if and only if $\mathcal{BA}\models\varepsilon\approx\tau $ and $\Var(\varphi)=\Var(\tau)$. 
Involutive bisemilattices can be connected to semilattice direct systems, in a way that we sketch. It is always possible to construct an algebra out of a semilattice direct system in an algebraic category. 
The construction we have in mind is called \emph{P\l onka sum} and is due to J. P\l onka \cite{Plo67}. For standard information on P\l onka sums we refer the reader to \cite{Romanowska92}.

\begin{definition}\label{def: somma di Plonka}

Let $ \mathbb{A}=\pair{\A_i , p_{ii'}, I} $ be a semilattice direct system of algebras  of a fixed type $\nu$. The \emph{P\l onka sum} over $\mathbb{A}$ is the algebra $\PLA = \pair{\bigsqcup_I A_i, g^{\mathcal{P}}} $, whose universe is the disjoint union of the algebras $\A_i$ and the operations $ g^{\mathcal{P}} $ are defined as follows: for every $n$"-ary $g\in \nu$, and $a_1,\dots, a_n\in \bigsqcup_I A_i $, where $n\geq 1$ and $a_r\in A_{i_r}$, we set $j = i_1 \lor\dots\lor i_n$ and define\footnote{In case $\nu$ contains constants, then, it is necessary to assume that $I$ has a least element, see \cite{plonka1984sum} for details.}
\[
g^{\mathcal{P}} (a_1,\dots,a_n) = g^{\A_j} (p_{i_1j}(a_1),\dots,p_{i_nj}(a_n)).
\]

\end{definition}

P\l onka sums provide a useful tool to represent algebras belonging to regular varieties. We recall here the representation theorem for involutive bisemilattices.

Involutive bisemilattices, as well as bisemilattices admit a representation as P\l onka sums over a semilattice system of Boolean algebras.
%
From \cite[Thm.~46]{Bonzio16} we know that, 
if $\mathbb{A} $ 
is a semilattice direct system of Boolean algebras, then $ \PLA $ is an involutive bisemilattice,
and if $\B$ is an involutive bisemilattice, then $\B\cong\PLA$, where $\mathbb{A}$ is a semilattice direct system of Boolean algebras.
The above facts can be strengthened to a full categorical equivalence.

\begin{theorem}[\mbox{\cite[Thm.~4.5]{Loi}}]\label{th: IBSL e Sem-dir-BA sono equivalenti}
The categories $\mathfrak{IBSL}$ and \emph{Sem-dir-}$\mathfrak{BA}$ are equivalent.
\end{theorem}

The equivalence is proved by the functors associating to each involutive bisemilattice the semilattice direct system of Boolean algebras corresponding to its P\l onka sum representation. Conversely, to each semilattice direct system (of Boolean algebras), it is associated the P\l onka sum. 
Upon considering Stone duality \cite{Stone37} between Boolean algebras and Stone spaces, $\mathfrak{SA}$ for short, namely compact totally disconnected Hausdorff topological spaces (see e.g. \cite{J82}), we have that 

\begin{theorem}[\mbox{\cite[Thm.~4.6]{Loi}}]\label{th:finibsldual}
The categories $\mathfrak{IBSL}$ and \emph{Sem-inv-}$\mathfrak{SA}$ are dually equivalent.
\end{theorem}

For the purpose of the present work, we restrict Stone duality to the finite setting, which reduces to the well-known duality between the category of finite Boolean algebras and their homomorphisms $\mathfrak{FBA}$,
and the category of finite sets and set-functions $\mathfrak{FS}$.
The functor implementing such duality maps any Boolean algebra $\A$ to the set $\widehat{A}$ given by the atoms of $\mathbf{A}$.
Thanks to the previous considerations, the problem of counting finite involutive bisemilattices coincides with counting semilattice direct systems (with finite index set) of finite Boolean algebras.

\section{Linearly ordered \IBSL}\label{sez: classe l-ordered}

We confine our concern to a specific class of involutive bisemilattices, namely those ones whose corresponding direct system has a linearly ordered index set. We call this class
\emph{linearly ordered} involutive bisemilattices, $\LIBSL$ for short.  
\begin{remark}
The class $\LIBSL$ is closed under subalgebras and homomorphic images but not under products. Closure under subalgebras is obvious. For homomorphic images, it is enough to observe that any homomorphism between elements in $\LIBSL$ corresponds to a morphism between the equivalent semilattice direct systems, which, restricted to the index sets, is a homomorphism of semilattices. Moreover, any homomorphic image of a totally ordered semilattice is totally ordered. As regards products, it is clear that the product of (semilattice direct) systems whose index set is linearly order may, in general, be a system whose index set is not linearly ordered.
\end{remark}

In the following part, we provide a first-order characterization of the class $\LIBSL$. Recall from \cite{Bonzio16} that, given an \emph{involutive bisemilattice} $\B$, an element $a\in B$ is called \emph{positive} if $a\vee \neg a = a$. We denote by $P(\B)$ the set of positive elements of $\B$.
\begin{remark}\label{rem: i positivi sono 1 booleani}
The set $P(\B)$ of positive elements of an involutive bisemilattice $\B$ coincides with the set of the constants $1$ of each Boolean algebra in the P\l onka decomposition of $\B$. In other words, 
$\B\cong\PLA$, with $\mathbb{A}=\pair{\A_i , p_{ii'}, I} $  a semilattice direct system of Boolean algebras, then $P(\B)=\bigcup_{i\in I}\{1_i\}$, where $1_i$ denotes the element $1$ in the Boolean algebra $A_i$. Checking that $1_i$ is positive for each $i\in I$ is immediate. On the other hand, suppose that $a\in P(\B)$, i.e. $a\vee \neg a = a$. Clearly, $a\in A_i$, for some $i\in I$, hence $\neg a\in A_i$. Therefore $ a= a\vee^{\PL}\neg a = a\vee^{\A_i}\neg a = 1_{i}$. 
\end{remark}

In the following result, we refer (with a slight abuse of notation) to the semilattice of indexes of a P\l onka sum as $\pair{I,\leq}$ and to the semilattice formed by the positive elements of an involutive bisemilattice with respect to the reducts $\wedge $ ($\vee$, respectively) as $\pair{P(\B),\wedge}$ ($\pair{P(\B),\vee}$, respectively).

\begin{proposition}\label{prop: indici isomorfi ai positivi}
Let $\B\in\IBSL$ and let $P(\B)$ be the set of positive elements. Then
 $$ 1.~ \pair{P(\B), \land}\cong\pair{P(\B), \vee}; \qquad
   2.~ \pair{P(\B), \vee}\cong\pair{I,\leq}.$$
\end{proposition}
\begin{proof}
\begin{enumerate}
\item The isomorphism is given by the identity map. We just check that, for any $a,b\in P(\B)$, $a\land b = a\vee b$. In virtue of Remark \ref{rem: i positivi sono 1 booleani}, we can assume that $a=1_i$ and $b=1_j$, for some $i,j\in I$. Let $k= i\vee j$, then: 
$a\wedge b = 1_i \wedge^{\PL} 1_j  = p_{ik}(1_i)\wedge^{\A_k} p_{jk}(1_j)  = 1_k \wedge^{\A_k} 1_k = 1_k = 1_k \vee^{\A_k} 1_k = p_{ik}(1_i)\vee_{\A_k} p_{jk}(1_j)  = 1_i \vee^{\PL} 1_j  = a\vee b.$

\item We again assume the identification highlighted in Remark \ref{rem: i positivi sono 1 booleani}. Consider the map $f\colon I\to P(\B)$, defined as $f(i)\coloneqq 1_i$. The map is invertible (with inverse $g\colon P(\B)\to I$, $g(1_i)=i$). Moreover, we check that $f$ is a homomorphism (of semilattices). Let $i,j\in I$ and $i\vee j = k$. Then 
$f(i\vee j) = f (k) = 1_k = 1_k \vee 1_k = p_{ik}(1_i) \vee p_{jk}(1_j) = 1_i \vee 1_j = f(i)\vee f (j).$
\qed
\end{enumerate}
\end{proof}

The above results shows that any consideration about the index set of the P\l onka sum representation of an involutive bisemilattice can be expressed over the (partially ordered) set of its positive elements. This turns out to be convenient since the subset of positive elements is equationally definable. Observe, moreover that the two partial orders induced by the binary operations of an involutive bisemilattice (see \cite{Bonzio16} for details) coincide over the set of positive elements.

\begin{corollary}\label{cor: descrizione I-IBSL}
Let $\B\in\IBSL$ and $\pair{P(\B),\leq}$ the poset of its positive elements. The following are equivalent:
$$ 1.~ \B\in\LIBSL;\qquad 2.~ \text{either }x\leq y\text{ or }y \leq x,\text{ for any }x,y\in P(\B).$$
\end{corollary}

We provide a useful criteria to detect isomorphic copies of linearly ordered involutive bisemilattices.

\begin{lemma}\label{lem:iso1}
Let $\mathbb{A}=\langle \A_{i}, p_{ii'}, I\rangle$ and $\mathbb{B}=\pair{\B_i, q_{ii'}, I}$ be two finite semilattice direct systems of Boolean algebras, with $I$ linearly ordered and containing no trivial algebras. Then, the following statements are equivalent: 
\begin{enumerate}
\item $\mathbb{A}\cong\mathbb{B}$ 
\item$\A_i\cong\B_i$, for every $i\in I$, and $\mid\widehat{A}_{i'}/ker (\widehat{p}_{ii'})\mid\; =\;\mid \widehat{B}_{i'} /ker (\widehat{q}_{ii'})\mid $, for every $i < i'$. 
\end{enumerate}
\end{lemma}
\begin{proof}
($\Rightarrow$) Assume that $\mathbb{A}\cong\mathbb{B}$ via an isomorphism $(\varphi , f_{i})$, for each $i\in I$. Since $\mathbb{A}$ and $\mathbb{B}$ share the same, linearly ordered, index set$I$, we necessarily have that $\varphi=id$. 
Moreover, for each $i\in I$, $\A_i\cong\B_i$ via the Boolean isomorphism $f_i$, and the following diagram on the left 
(we deliberately drop indexes from $p$ and $q$ to make notation less cumbersome) is commutative, for each $i < i'$,
and in virtue of the duality established in Theorem~\ref{th:finibsldual} and the definition of inverse systems, 
the following diagram on the right is also commutative:


\begin{center}\leavevmode
\xymatrix{
 \mathbf{A}_i\ar@{->}[r]^-{p}\ar@{->}[d]_-{f_{i}} & \mathbf{A}_{i'}\ar@{->}[d]^-{f_{i'}} &&\\
 \mathbf{B}_{i} \ar@{->}[r]^-{q} 	& \mathbf{B'}_{i} &&
	}
\xymatrix{
 \widehat{A}_i\ar@{->}[d]_-{\widehat{f}^{-1}_{i}} & \widehat{A}_{i'}\ar@{->}[l]_-{\widehat{p}}\ar@{->}[d]^-{\widehat{f}^{-1}_{i'}}\\
 \widehat{B}_{i} 	& \widehat{B'}_{i} \ar@{->}[l]^-{\widehat{q}}
	}
\end{center}

By the first isomorphism theorem (for sets), we have that there exist two (unique) embeddings $\psi\colon \widehat{A}_{i'}/ker(\widehat{p}\;)\to\widehat{A}_i $ and $\chi\colon \widehat{B}_{i'}/ker(\widehat{q}\;)\to\widehat{B}_i$ such that $\widehat{p}=\psi\circ\pi_{\widehat{A}_{i'}}$ and $\widehat{q}=\chi\circ\pi_{\widehat{B}_{i'}}$, where $\pi_{\widehat{A}_{i'}}$ and $\pi_{\widehat{B}_{i'}}$ indicate the natural projections onto the quotients $\widehat{A}_{i'}/ker(\widehat{p}\;)$, $\widehat{B}_{i'}/ker(\widehat{q}\;)$, respectively. The above diagram can therefore be split into the following:

\begin{center}\leavevmode
\xymatrix@R=2pc@C=6pc{
\widehat{A}_{i'}\ar@{->}[r]^-{\widehat{f}^{-1}_{i'}}\ar@{->}[d]_-{\pi_{\widehat{A}_{i'}}}&\widehat{B}_{i'}\ar@{->}[d]^-{\pi_{\widehat{B}_{i'}}} \\
\widehat{A}_{i'}/ker(\widehat(p)\ar@{-->}[r]^-{\Phi}\ar@{->}[d]_-{\psi}&\widehat{B}_{i'}/ker(\widehat{q})\ar@{->}[d]^-{\chi} \\
 \widehat{A}_i\ar@{->}[r]_-{\widehat{f}^{-1}_{i}} & \widehat{B}_{i}
	}
\end{center}

We define the map $\Phi\colon\widehat{A}_{i'}/ker(\widehat{p}\;) \to \widehat{B}_{i'}/ker(\widehat{q}\;) $ 
as $\Phi([a]_{p})\coloneqq [\widehat{f}^{-1}_{i'}(a)]_{q}$.

We claim that $\Phi$ is a bijection and this would conclude this part of the proof.

\noindent
In order to show the claim, let $[b]\in\widehat{B}_{i'}/ker(q) $, then $[b]=\pi_{\widehat{B}_{i'}}(b)$, for some $b\in\widehat{B}_{i'}$. By surjectivity of $\widehat{f}^{-1}_{i'}$, there exists an element $a\in\widehat{A}_{i'}$ such that $b=\widehat{f}^{-1}_{i'}(a)$. Therefore $[b]_{q}=[\widehat{f}^{-1}_{i'}(a)]_{q} = \Phi([a]_{p})$, i.e. $\Phi$ is surjective. 
To show that $\Phi$ is also injective, assume $[a]_{p}\neq [b]_{p}$, i. e. $a\neq b$, with $a,b\in \widehat{A}_{i'}$. Suppose, in view of a contradiction, that $\Phi([a]_p)=\Phi([b]_{p})$, i. e. $[\widehat{f}^{-1}_{i'}(a)]_{q}=[\widehat{f}^{-1}_{i'}(b)]_{q}$. By assumption and the fact that $\psi$ is an embedding, we have that $\psi\circ \pi_{\widehat{A}_{i'}}(a)\neq \psi\circ \pi_{\widehat{A}_{i'}}(b)$, i.e. $p(a)\neq p(b)$, whence $\widehat{f}^{-1}_{i}\circ p (a)\neq \widehat{f}^{-1}_{i}\circ p(b) $, since $\widehat{f}^{-1}_{i}$ is a bijection. On the other hand, 
$q\circ \widehat{f}^{-1}_{i'} (a)  =  \chi\circ\pi_{\widehat{B}_{i'}}\circ \widehat{f}^{-1}_{i'}(a) = \chi\circ\pi_{\widehat{B}_{i'}}\circ \widehat{f}^{-1}_{i'}(b) = q\circ \widehat{f}^{-1}_{i'} (b),$
which is in contradiction with the commutativity of the above diagram. This shows our claim. 
\vspace{5pt}

\noindent
($\Leftarrow$) Assume that $\A_i\cong\B_i$, by the family of isomorphisms $f_i\colon \A_i\to\B_i$, for each $i\in I$, and that there exists a bijection $\Phi\colon\widehat{A}_{i'}/ker(\widehat{p}\?) \to \widehat{B}_{i'}/ker(\widehat{q}\?) $. Then, it is easy to check that $(id, f_i)$ gives the desired isomorphism. 
\qed
\end{proof}

\begin{example}
The two linearly ordered involutive bisemilattices in the following picture (lines indicate orders in the Boolean components, dashed lines indicate Boolean homomorphisms) are isomorphic (this is an 
consequence of Lemma \ref{lem:iso1}).
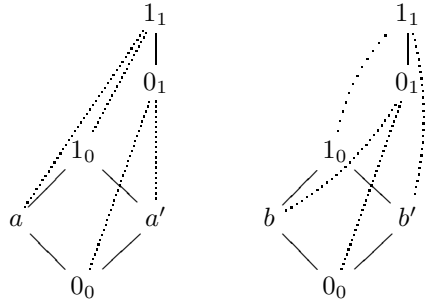
\begin{center}\leavevmode
\xymatrix@R=1pc@C=1pc{
			&			&{1_{1}}\ar@{..}[dddll]&&			&		&{1_{1}}\ar@/^0.5pc/@{..}[ddd]&\\
			&			&{0_{1}}\ar@{-}[u]\ar@{..}[dd]&&		&			&{0_{1}}\ar@{-}[u]	&\\
			&{1_{0}}\ar@{..}[uur]&				&&			&{1_{0}}\ar@/^0.5pc/@{..}[uur]&		&\\
{a}\ar@{-}[ur]	&			&{a'}\ar@{-}[ul]		&&{b}\ar@{-}[ur]\ar@/_0.5pc/@{..}[uurr]&	&{b'}\ar@{-}[ul]	&\\
			&{0_{0}}\ar@{-}[ur]\ar@{-}[ul]\ar@{..}[uuur]&	&&			&{0_{0}}\ar@{-}[ur]\ar@{-}[ul]\ar@{..}[uuur]&&
	}
\end{center}
\noindent
In order to exhibit a concrete isomorphism, observe that the two algebras are constructed using the very same index set (the two element lattice).
An isomorphism is given by considering the identity map $id$ on the lattice of indexes, the unique isomorphism on the two elements Boolean algebras and the Boolean isomorphism given by $\varphi(a)= \neg b$ (and $\varphi(\neg a)= b$) between the 4-elements Boolean algebras. 
\end{example}
The graphical convention adopted in the above example (dotted lines for homomorphisms between algebras in the P\l onka sum,
black lines for the usual order relation in a Boolean algebra) will be used throughout the paper.

\begin{remark}
Notice that the non-triviality assumption in Lemma \ref{lem:iso1} is crucial as witnessed by the following example, 
where we consider the two semilattice direct systems 
$\pair{\{\A_0, \mathbf{1}_1, \mathbf{1}_2\}, p_{ii'} , \{0,1,2\}}$ and $\pair{\{\A_{0},  \mathbf{1}_1,\mathbf{1}_2\}, q_{ii'} , \{0,1,2\}}$. 

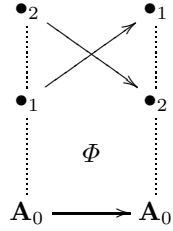
\begin{center}\leavevmode
\xymatrix@R=0.7pc@C=0.7pc{
 & \bullet_2\ar@{->}[ddrr] &  & \bullet_1 &  \\
 & &  & & \\
 & \bullet_1\ar@{..}[uu]\ar@{->}[uurr]&  & \bullet_2\ar@{..}[uu]  & \\
 & & \Phi & & \\
 & \A_0 \ar@{..}[uu]\ar@{->}[rr]&  & \A_0\ar@{..}[uu] & 
	}
\end{center}

The map $\Phi$ depicted between the two systems is not an isomorphism (as it is clearly not an isomorphism on the index set), however $1_{2}/ker (\widehat{p}_{_{12}}) = 1_{2} = 1_{2} / ker (\widehat{q}_{_{12}})$.

\end{remark}

We denote by $\mathbb{A}_{p_{i-1,i}} $ the family $\{\pair{\A_{i}, p_{i-1,i}, I}\}_{i\in I} $ of all finite semilattice direct systems of Boolean algebras obtainable from the family of algebras $\{\A_i\}_{i\in I}$, indexed over the linearly ordered index set $I$. When clear from the context, we will write $\mathbb{A}_p$ instead of $\mathbb{A}_{p_{i-1,i}} $.

We are interested in the following question: what is the number of non-isomorphic elements in the family $\mathbb{A}_{p}$? This, in turn, will provide an answer to the question about how many linearly ordered involutive bisemilattices have the cardinality $\bigcup_{i\in I}\mid A_{i}\mid$, up to isomorphism. 

In the light of Theorem~\ref{th:finibsldual} and Lemma \ref{lem:iso1}, the answer to is provided by considering 
the dual semilattice inverse systems $\widehat{\mathbb{A}}_{\hat{p}}$. 

\begin{lemma}\label{lem: conteggio hom con isofilter}
Let $\mathbb{A}_{p_{_{01}}}=\pair{\{\A_0, \A_1\}, p , \{0<1\}}$ be a family of linearly ordered semilattice direct system of Boolean algebras. The number of non-isomorphic involutive bisemilattices obtained over $\mathbb{A}_{p_{01}}$ is
the number of non-isomorphic involutive bisemilattices of cardinality $\mid A_0 \mid + \mid A_1 \mid$ and is equal to
\[
N(\mathbb{A}_{p_{_{01}}})\coloneqq N(A_0, A_1)=\min (\mid \widehat{A}_0\mid, \mid\widehat{A}_1\mid),
\]
where $\widehat{A}_0$ ($\widehat{A}_{1}$, resp.) is the dual space of $\A_0$ ($\A_1$, resp.)

\end{lemma}
\begin{proof}
At first observe that two elements in the family $\mathbb{A}_{p_{_{01}}}$ differ only for the Boolean homomorphism from $\A_0$ to $\A_1$. Therefore, by Lemma \ref{lem:iso1}, two involutive bisemilattices constructed over the system $\mathbb{A}_{p_{_{01}}}$ are not isomorphic if and only if $\mid\widehat{A}_{1}/ker (\widehat{p}\?)\mid\; \neq \;\mid \widehat{A}_{0} /ker (\widehat{q}\?)\mid $ (where $p$ and $q$ are the Boolean homomorphisms). In other words, this means that the kernels of $\widehat{p}$ and $\widehat{q}$ generate two partitions, over $\mid\widehat{A}_1\mid$, with a different number of equivalence classes. It is known that the number of equivalence classes partitioning a finite algebra $\A$ into a different number of blocks is equal to $\mid A \mid$. Therefore, if $\mid A_1\mid\leq \mid A_0\mid$ then $ N(A_0, A_1)=\;\mid \widehat{A}_1\mid $. Differently, since we only consider (the number of) partitions induced by (kernels of) maps from $\mid\widehat{A}_1 \mid$ to $\mid\widehat{A}_0 \mid$, we have that $N(A_0, A_1)=\;\mid \widehat{A}_0\mid$.
\qed
\end{proof}

\begin{remark}\label{rem:isopart2}
It is easily checked that the function $N(\mathbb{A}_{p_{01}}) $, counting the number of involutive bisemilattices obtained over $\mathbb{A}_{p_{01}} $, can be generalized to the family $\mathbb{A}_{p_{_{0m}}} $ of semilattice direct systems of Boolean algebras $\pair{\{A_{0},\dots, A_{m}\}, p_{i-1,i}, I}$. More precisely, 
\[
N(\mathbb{A}_{p_{_{0m}}}) = N(\mathbb{A}_{p_{_{01}}}) \cdot N(\mathbb{A}_{p_{_{12}}}) \cdot_{\dots} \cdot N(\mathbb{A}_{p_{_{m-1 m}}}) = \prod_{i=0}^{m-1} N(\mathbb{A}_{p_{_{ii+1}}}).
\]
\begin{figure}[h!]
\begin{center}\leavevmode
\xymatrix@R=0.7pc@C=0.7pc{
*={\bullet}\ar@{..}[d] 	&			&				&\\
*={\bullet}\ar@{..}[d] 	&			&				&\\	
*={\bullet}\ar@{..}[d] 	&			&				&*{\bullet}\ar@{-}[d]\ar@{..}[dl]\\
*{\bullet}\ar@{..}[d] 	&			&*{\bullet}\ar@{-}[d]\ar@{..}[dl]	&*={\bullet}\\
*{\bullet}\ar@{..}[d]   &*{\bullet}\ar@{-}[d]	&*={\bullet}\ar@{..}[ur]&\\
*={\bullet}		&*={\bullet}\ar@{..}[ur]&			&\\
	}
\xymatrix@R=0.7pc@C=0.7pc{
	&	&	&&*{\bullet}\ar@{..}[dl]		&			&			\\
	&	&	&*{\bullet}\ar@{..}[dl]			&		&	&*={\bullet}\ar@{..}[dl]\\	
	&	&*{\bullet}\ar@{..}[dl]	&			&			&*={\bullet}\ar@{..}[dl]\\
	&*{\bullet}\ar@{..}[dl]	&	&			&*{\bullet}\ar@{-}[d]\ar@{..}[dl]&&\\
*{\bullet}\ar@{-}[d]	&	&	&*{\bullet}\ar@{-}[d]	&*={\bullet}\ar@{..}[dl]	&&\\
*={\bullet}	&	&		&*={\bullet}		&			&&\\
	}
\xymatrix@R=0.7pc@C=0.7pc{
	&&					&			&	&			&\\
	&		&			&*={\bullet}		&	&			&	\\
	&		&*={\bullet}\ar@{..}[ur]&			&	&*={\bullet}\ar@/_0.5pc/@{..}[ddll]	&		\\
	&*={\bullet}\ar@{..}[ur]&		&			&*={\bullet}\ar@{-}[ur]&	&*={\bullet}\ar@{-}[ul]\\
*={\bullet}\ar@{-}[ur]&	&*={\bullet}\ar@{-}[ul]	&*={\bullet}		&	&*={\bullet}\ar@{-}[ur]\ar@{-}[ul]&\\
	&*={\bullet}\ar@{-}[ur]\ar@{-}[ul]&	&*={\bullet}\ar@{-}[u]\ar@{..}[urr]&	&	&\\
	}
\xymatrix@R=0.7pc@C=0.7pc{
			&			&			&\\	
			&			&*={\bullet}\ar@{..}[dddll]&\\	
			&			&*={\bullet}\ar@{-}[u]\ar@{..}[dd]& \\	
			&*={\bullet}\ar@{..}[uur]&			&\\		
*={\bullet}\ar@{-}[ur]	&			&*={\bullet}\ar@{-}[ul]	&\\
			&*={\bullet}\ar@{-}[ur]\ar@{-}[ul]\ar@{..}[uuur]&
	}
\end{center}
\caption{The linearly ordered non-isomorphic IBSLs of cardinality 6.}
\label{fig:finibsl}
\end{figure}
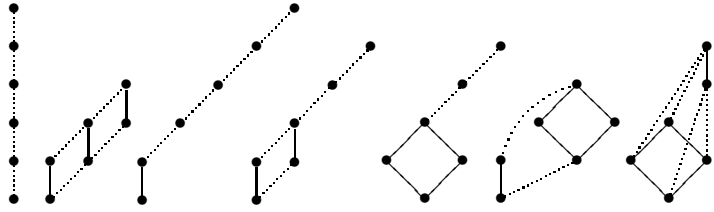
\end{remark}

\section{Generating and counting $\LIBSL$}\label{sec:algos}

The previous section provides sufficient and necessary conditions in order to identify the non-isomorphic \LIBSL obtained over direct systems sharing both the index set and the family of Boolean algebras. In general, the elements of such a family of direct systems  $\mathbb{A}_{p}$  differ at most with respect to  the definition of their homomorphisms. It is easy to check that, given $\mathbb{X}=\langle \A_{i},p_{ii^{\prime}}, I\rangle$ and $\PL(\mathbb{X})\in\LIBSL$, we can always single out the family of direct systems $\mathbb{A}_{p}=\{\langle \A_{i},p_{i-1,i}, I\rangle\}$. As this fact is central in the rest of the paper, we introduce the following definition


\begin{definition}\label{def shape}
Let $\mathbb{X}= \pair{\A_{i}, p_{ii'}, I} $ be a semilattice direct system of Boolean algebras and $\B\cong\PL(\mathbb{X})\in\LIBSL$. We define the \emph{shape} of $\B$ as $\mathbb{X}_{\B}=\pair{\A_{i}, I}$.

\end{definition}

Clearly, two $\LIBSL$, $\A\cong\PL(\mathbb{X}), \B\cong\PL(\mathbb{Y})$  have the same shape if and only if  the  semilattice direct systems $\mathbb{X}$ and $\mathbb{Y}$ differ at most with respect to their homomorphisms, i.e. they belong to the same family $\mathbb{A}_{p}$. So, with this terminology at hand, Lemma \ref{lem: conteggio hom con isofilter} and Remark \ref{rem:isopart2} tell us the number of non-isomorphic $\LIBSL$ of a fixed shape. Moreover, as a consequence of Lemma \ref{lem:iso1}, two \LIBSL with different shapes are non-isomorphic. 

Therefore, in order to answer our  question concerning the number of finite algebras in $\LIBSL$ it only remains to count, for a given $n\in\mathbb{N}$, the number of shapes that an $\LIBSL$ of order $n$ can have. The present section is devoted to this issue.

Let $n\in\mathbb{N}^+$, and $e, m_i\in \mathbb{N}$ for $0\leq i\leq e$. 
A \emph{binary partition} of $n$ is a decomposition of $n$ into powers of two, that is
\begin{equation}\label{eq:2dec}
n = m_{e}\cdot 2^{e}+m_{e-1}\cdot 2^{e-1} +\dots+ m_{0}\cdot 2^{0}.
\end{equation}
Hence, the number $b(n)$ of binary partitions\footnote{See sequence http://oeis.org/A018819
at \emph{The On-Line Encyclopedia of Integer Sequences}, published electronically at https://oeis.org.}
of $n$ is the number of solutions of \eqref{eq:2dec}.
In general, binary partitions which differs by the orders of summands are considered identical.

Knowing that each finite linearly ordered \IBSL $\B$ can be decomposed as a P{\l}onka sum of Boolean algebras $\mathbf{A}_i$ whose direct system is indexed by a totally ordered set $I$, it follows that the cardinality $n$ of $\B$ is always given by a solution of a binary partition \eqref{eq:2dec} where $2^i=|\mathbf{A}_i|$ and $|I|=\sum_{i=0}^e m_i$.

The fact that binary partitions cannot differ only for the order of summands, together with (iii) in Definition~\ref{def: direct system} and Remark~\ref{rem:isopart2}, implies  that $b(n)$ cannot account for the number of shapes that an $\LIBSL$ of order $n$ can assume. Indeed, given a certain shape $\langle \A_{i}, I\rangle$ of $\B\in\LIBSL$, every permutation $\phi$ over $I$ that moves at least two indexes $i,j$ such that $\mid\A_{i}\mid \ \neq  1, \mid\A_{j}\mid \ \neq  1$ defines a new shape $\langle \A_{i}, \phi(I)\rangle$.  Notice that the condition about $\mid\A_{i}\mid \ \neq 1, \mid\A_{j}\mid \ \neq 1$ is justified by the fact  that for any non-trivial Boolean algebra $\A$ there are no homomorphisms from a trivial Boolean algebra to $\A$. Hence, we can remove from the permutations $\phi(I)$ of $I$ all the algebras whose cardinality is $2^0$. The number of such algebras is expressed in \eqref{eq:2dec} as $m_0$. 
Now, looking at $\langle \A_{i}, I\rangle$ and $\langle \A_{i}, \phi(I)\rangle$ as binary partitions, it is immediate to observe that they only differ for the order of summands.

Define $I^+=I\setminus\{0\}$. For these reasons, we have to consider the \emph{permutations with repetitions} of 
$|I |-m_{0}=\sum_{i\in I^+} m_{i}$,
knowing that each $\mathbf{A}_i$ is repeated $m_i$ times.
The number of such permutations is given by the \emph{multinomial coefficient} \cite{GKP89}
\begin{equation}\label{eq:permrep}
pr(|I |-m_{0})=   {{|I |-m_{0}}\choose{ \{m_i\}_{i\in I^+}} }  = {\frac{(\sum_{i\in I^+} m_{i})!}{ \prod_{i\in I^+} m_{i}!}},
\end{equation}

We now start by introducing a routine procedure that generates all the binary partitions of $n$
in a form that makes the ensuing computations easy to handle.

\begin{definition}
A \emph{sequence} is a list of pairs of the form $s=(m_{e_{1}},2^{e_{1}})\to\dots\to(m_{e_{p}},2^{e_{p}})$.
We define its \emph{presentation} as
$$P(s) =  (2_1^{e_{1}}\to\dots\to2_{m_{e_{1}}}^{e_{1}})\to\dots\to(2_{1}^{e_{p}}\to\dots\to 2^{e_{p}}_{m_{e_{p}}}) .$$
\end{definition}

Given a set $S$ of sequences, we denote by $P(S)$ the set of the presentations of sequences in $S$,
that is $P(S)=\{P(s)\mid s\in S\}$.

\begin{definition}\label{def:seqpart}
Given a natural number $n\in\mathbb{N}^+$ we define 
$$L(n)=\{(m_{e_1},2^{e_1})\to(m_{e_2},2^{e_2})\to\dots\to(m_{e_k},2^{e_k})\mid n=\sum_{i=1}^k (m_{e_i}\cdot 2^{e_i})\},$$
such that $e^1>e^2>\dots>e^k$, 
that is the set of sequences that give all decompositions of $n$ into powers of two, such as in \eqref{eq:2dec}.

For any sequence $l(n)\in L(n)$, define 
$F(l(n))=\{m_{e_i}\mid (m_{e_i},2^{e_i})\in l(n)\}$ as the set of the \emph{multiplicities} in $l(n)$,
and $F^{+}(l(n))=\{m_{e_i}\in F(l(n))\mid e_i\not=0\}$ 
as the subset of  $F(l(n))$ given by the \emph{positive multiplicities} in $l(n)$.
\end{definition}

Observe that the above definition of multiplicity  and positive multiplicity in $l(n)$ can be equally defined by looking at the presentation $P(l(n))$. Moreover, given a sequence $l(n)=(m_{e_{1}},2^{e_{1}})\to\dots\to(m_{e_{p}},2^{e_{p}})$ with positive factors $m_{e_{1}},...,m_{e_{z}}$, we denote by $P^{+}(l(n))$ the presentation of the sequence $(m_{e_{1}}, 2^{e_{1}})\to...\to(m_{e_{z}}, 2^{e_{z}})$.

\begin{definition}\label{def:division map}
Let $n, n_{1},n_{2}\in\mathbb{N} $ and $ E(n)=\{ 2^e,\dots,2^0\}$ the set of powers of two such that $2^i \leq \log_{2} n$, for any $i= 0,\dots, e$. The map $d\colon\mathbb{N} \times\mathbb{N} \to E(n)$  defined as 
\[
d(n_{1},n_{2})\coloneqq \left\{ \begin{array}{ll}
max\{m\in E(n):m<n_{1},n_{2}\}& \text{ if  $n_{1}, n_{2}\not\in\{0,1\}$}\\
 1 & \text{otherwise.}
  \end{array} \right.  
\]
is called the \emph{division map} of $(n_{1},n_{2})$ with respect to $n$.
\end{definition}

By a forest we mean a disjoint union $\sqcup$ of trees.

Given two sequences $l(n)=(m_{e_1},2^{e_1})\to\dots\to(m_{e_k},2^{e_k})$, 
$l'(n)=(m_{f_1},2^{f_1})\to\dots\to(m_{f_h},2^{f_h})$ 
in $L(n)$, we say that they share a \emph{common prefix} when
for some $i\leq min(k,h)$, we have $e_j= f_j$, for $1\leq j\leq i$.
We write $l(n)=P\to s$ and $l'(n)=P\to s'$
to denote the fact that $P$ is the common prefix of $l(n)$ and $l'(n)$.

Given a set of sequences $S$, we construct a forest $\Gamma(S)$ in the following way.
Let $l$ and $l'$ be two chains in $S$, such that $P$ is their longest common prefix,
that is $l=P\to s$ and $l'=P\to s'$.
Then, the tree $P\to(s\sqcup s')$ belongs to $\Gamma(S)$.
For each sequence $s$ in $S$ there is a unique branch of $\Gamma(S)$ that is a unique copy of $s$,
and every branch of $\Gamma(S)$ is a copy of a unique chain in $S$.

The pseudocode in Algorithm~\ref{alg:gen} introduces a couple of functions that, for any given $n\in\mathbb{N}^+$, 
produce the set $\Gamma(L(n))$, that is the set of trees whose branches are binary partitions of $n$ 
expressed as sequences of Definition~\ref{def:seqpart}.
Notice that, to better follow the construction of sequences of $\Gamma(L(n))$,
in the pseudocode of Algorithm~\ref{alg:gen} we repeatedly use expressions like $2^e$.
Obviously, a real-world implementation of the algorithm does not need such a level of detail,
and exponents $e$ can be used instead.

\begin{algorithm}
\caption{}\label{alg:gen}
\begin{algorithmic}[1]
\Function{GenForest}{$n$}
	\State $e=\lfloor \log_2 n \rfloor$
	\State $E= \{2^{e},2^{e-1},\dots, 2^0 \}$ \label{line:expset}
	\State $F$ empty list of trees
	\ForEach{$2^e$ in $E$} 
		\State $T_e= $ \Call{GenSeq}{$n,2^e$}
		\Comment{$T_e$ is a list of trees.}
		\State add $T_e$ to $F$
	\EndFor
	\State \textbf{return} $F$
\EndFunction
\Statex

\Function{GenSeq}{$n$,$2^e$} 
	\If{$e==0$}
		\State\textbf{return} a tree with root $(n,2^0)$ \label{line:zero}
	\EndIf
	\State $q= n/2^e$ 
	\If{$q>1$}
		\ForEach{$i\in\{1,\dots,q\}$} 
			\State create a tree $P_i$ with root $(i, 2^e)$	\label{line:start}
			\State $m=n-i\cdot 2^e$
			\If{ $m>0$}
			\If{$m=2^x$ and $x<e$} 
				\State $d_i=m$
			\Else
				\State $d_i = d(m, 2^e)$
				\Comment{$d$ refers to the division map of Definition~\ref{def:division map}}
			\EndIf
				\ForEach{$2^j$ in $\{2^0,\dots, d_i\}$}
					\State $T= $ \Call{GenSeq}{$n-i\cdot 2^e, 2^j$}	
					\Comment{$T$ is a list of trees.}
					\State for every $t\in T$, add $t$ as a child of $P_i$
				\EndFor
			\EndIf
		\EndFor\label{line:end}
	\EndIf
	\If{$q==1$}
		\State create a tree $P_i$ with root $(1,2^e)$ \label{line:resto}
		\State $r= n~\textbf{mod}~2^e$
		\If{$r>0$}
			\State $T= $ \Call{GenForest}{r}
				\Comment{$T$ is a list of trees.}
			\State for every $t\in T$, add $t$ as a child of $P_i$

		\EndIf
	\EndIf
	\State \textbf{return} $P_i$
\EndFunction
\end{algorithmic}
\end{algorithm}

\begin{theorem}\label{th: algoritmo completo}
$\Gamma(L(n))= \ $\textsc{GenForest}$(n)$.
\end{theorem}
\begin{proof}
$\supseteq$. This inclusion follows by direct inspection.

$\subseteq$. We prove this inclusion by an induction on $n$.

(B). $n=1$. If $n=1$ then $\Gamma(l(n))=\{(1,2^{0})\}$. As $E(1)=\{2^{0}\}$ by Line 6 the algorithm calculates \textsc{GenSeq}$(1,2^{0})$
 which by Line 13 has as output exactly the sequence $\{(1,2^{0})\}$.

(IND). Assume the statement holds for any $\{1,\dots, n\}$ and consider a sequence 
$l(n+1)=(m_{1},2^{e_{1}})\to...\to(m_{p},2^{e_{p}})\in\Gamma(l(n+1))$. 
Clearly $2^{e_{1}}\in E(n+1)$, then in \textsc{GenSeq}$(n+1,2^{e_{1}})$
at Lines 15 we obtain $(n+1)/2^{e_{1}}=q$. 
We have two cases (a) $m_{1}=1$ or (b) $m_{1}\gneq 1$. 
In the first case (a), Lines 17-18 entails that $(1,2^{e_{1}})$ is generated as root of the sequence in case $q>1$. 
Similarly, when  $q=1$, Lines 34-35 generate $(1,2^{e_{1}})$ as a root of the sequence.
For the second case (b), we have $m_{1}\in\{1,...,q\}$
and by Line 18  the algorithm gives $(m_{1},2^{e_{1}})$ as a root of the sequence.

Now, clearly $l(n+1)\smallsetminus(m_{1},2^{e_{1}})\in\Gamma(l((n+1)-m_{1}\cdot 2^{e_{1}}))$, 
i.e. $l(n+1)\smallsetminus(m_{1},2^{e_{1}})$ is a sequence for $(n+1)-m_{1}\cdot 2^{e_{1}}$. 
By induction hypothesis $l(n+1)\smallsetminus(m_{1},2^{e_{1}})\in$\textsc{GenForest}$((n+1)-m_{1}\cdot 2^{e_{1}})$. 
To complete the proof we have to show that the second pair $(m_{2},2^{e_{2}})$ in $l(n+1)$ is generated 
as child of the root $(m_{1},2^{e_{1}})$ is some branch.
In order to simplify the notation we fix $k=(n+1)-m_{1}\cdot 2^{e_{1}}$.

We distinguish two cases:

(1). $m_{2}=1$. Then either (a) $2^{e_{2}}=k$ or (b) $2^{e_{2}}\in\{2^{0},...,d(2^{1},k)\}$. 
In the first case (a) we have that $k$ is a power of $2$ strictly smaller than $2^{1}$, and  
by Line 21,25,26,27, follows that the computation of \textsc{GenSeq}$(k, 2^{0})$ 
returns $(1, k)$ as child of $(m_{1}, 2^{e_{1}})$ (by Lines 12-13). 
In the second case (b), by Lines 26,27 the algoritm computes \textsc{GenSeq}$(2^{e_{2}}, k)$. 
So, it is immediate to verify that for any $k/2^{e_{2}}$ computed at Line 15, 
by Lines 17,18 and 33,34 we obtain $(1,2^{e_{2}})$ as child of $(m_{1},2^{e_{1}})$.

(2). $m_{2}>1$. This implies $2^{e_{2}}<2^{e_{2}}$ and therefore $2^{e_{2}}\in\{2^{0},...,d(k,2^{e_{1}})\}$. 
By Line 26 we have the computation of \textsc{GenSeq}$(k,2^{e_{2}})$. 
Clearly $m_{2}\leq k/2^{e_{2}}$ and this, 
together with the assumption $m_{2}>1$ implies $k/2^{e_{2}}>1$. 
So, by Lines 17,18 $(m_{2},2^{e_{2}})$ is a child of $(m_{1},2^{e_{1}})$.

So, as $(m_{2},2^{e_{2}})$ is the root of the sequence $l(n+1)\smallsetminus(m_{1},2^{e_{1}})$ 
and by induction hypothesis  $l(n+1)\smallsetminus(m_{1},2^{e_{1}})$ is generated by the algorithm, 
the fact that $(m_{2},2^{e_{2}})$ is generated as child of $(m_{1},2^{e_{1}})$ entails that $l(n+1)$ 
is generated by the algorithm, as desired. \qed
\end{proof}

\begin{example}\label{exa:ex10}
For $n=10$ the output of \textsc{GenForest}$(10)$ is depicted in Figure~\ref{fig:gentree10}.
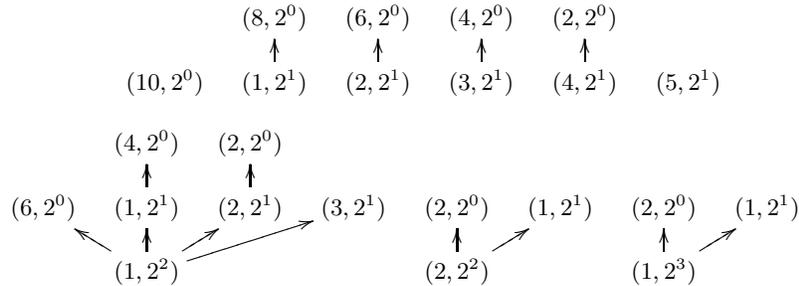
\begin{figure}[h!]
\begin{center}\leavevmode
\xymatrix@R=0.7pc@C=0.7pc{
	&		&			&		&			&		\\
	&(8,2^0)	&(6,2^0)		&(4,2^0)	&(2,2^0)		&		\\
(10,2^0)&(1,2^1)\ar@{->}[u]&(2,2^1)\ar@{->}[u]&(3,2^1)\ar@{->}[u]&(4,2^1)\ar@{->}[u]	&(5,2^1)
	}

\medskip

\xymatrix@R=0.7pc@C=0.7pc{
	&	&(4,2^0)						&(2,2^0)		&	&		& & & &\\
 	&(6,2^0)&(1,2^1)\ar@{->}[u]					&(2,2^1)\ar@{->}[u]	&(3,2^1)&(2,2^0)	&(1,2^1)		&(2,2^0)		&(1,2^1)	&\\
	&	&(1,2^2)\ar@{->}[u]\ar@{->}[ul]\ar@{->}[ur]\ar@{->}[urr]&			&	&(2,2^2)\ar@{->}[u]\ar@{->}[ur]	&&(1,2^3)\ar@{->}[u]\ar@{->}[ur]	&&	
}
\end{center}
\caption{The forest of sequences generated by \textsc{GenForest}$(10)$, see Example~\ref{exa:ex10}.}
\label{fig:gentree10}
\end{figure}
\end{example}

\begin{remark}\label{L-ibsl is  sequence}

It is worth noticing that, given a sequence $s=(m_{e_{1}},2^{e_{1}})\to\dots\to(m_{e_{p}},2^{e_{p}})$  
such that $e_{i}\neq 0$ for any $e_{1}\leq e_{i}<e_{p}$,  
its presentation $P(s) =  (2_1^{e_{1}}\to\dots\to2_{m_{e_{1}}}^{e_{1}})\to\dots\to(2_{1}^{e_{p}}\to\dots\to 2^{e_{p}}_{m_{e_{p}}})$ 
always describe a shape  $\mathbb{X}_{\B}$ of an \LIBSL $\B$. 
More precisely, $\mathbb{X}_{\B}= \pair{\A_{i}, I}$ 
where  $I=\{1_{e_{1}},\dots, m_{e_{1}}\}\cup\{1_{e_{2}},\dots, m_{e_{2}}\}\cup\dots\cup\{1_{e_{p}},\dots, m_{e_{p}}\}$ 
and for each $k_{e_{i}}\in I$, $\A_{k_{e_{i}}}$ is a Boolean Algebra of order $2^{e_{i}}$. 
Dually,
given an \LIBSL $\B\cong\PL(\mathbb{X})$, its shape $\mathbb{X}_{\B}$ 
can always be described by the presentation of an appropriate sequence. 
\end{remark}

\noindent
\textbf{Notation}. In what follows we adopt the following notation. Given $l(n)\in L(n)$, we denote by $\pi(l(n))=\{s_{1},\dots,s_{c(l(n))}\}$ the set of sequences obtained by a permutation with repetition over  $P^{+}(l(n))$.

\begin{example}
Let $l'(10)=(1,2^2)\to(2,2^1)\to(2,2^0)$ be a sequence in $L(10)$
(compare with the branches in the biggest tree in Figure~\ref{fig:gentree10}).
The only permutation in $P^{+}(l'(10))$ is $(2,2^1)\to(1,2^2)\to(2,2^0)$.
\end{example}

 \begin{theorem}\label{teorema principale}
Let $n\in\mathbb{N}^+$, $l(n)=(m_{e_{1}}, 2^{e_{1}})\to...\to(m_{e_{q}},2^{e_{q}})\in L(n)$ with  $k=m_{e_{1}}+...+ m_{e_{q}}$
and let also $m_{e_{1}},...,m_{e_{z}}$ be the members of $F^{+}(l(n))$. Then,  there are
 \[c(l(n))=\frac{(m_{e_{1}}+...+m_{e_{z}})!}{m_{e_{1}}!\cdot...\cdot m_{e_{z}}!}\]
shapes $\langle {I_{1}, \A_{i}^{1}\rangle,\dots, \langle I_{c(l(n))},\A_{i}^{c(l(n))}}\rangle$ of finite \LIBSL  of order $n$ 
such that, for each $1\leq j\leq c(l(n))$, $\mid I_{j}\mid=k$ and $\A_{i}^{j}$ is a family of Boolean Algebras whose cardinalities correspond to the members of $P(l(n))$. 
 
 \end{theorem}

\begin{proof}

Observe that the above formula 
is an instance of \eqref{eq:permrep}, 
counting the permutations with repetitions of a set of cardinality $m_{e_{1}}+...+m_{e_{z}}$ 
knowing that each $1\leq i\leq z$ object is repeated $m_{e_{i}}$ times.
By Remark \ref{L-ibsl is  sequence}, each $s\in\pi(l(n))$ describes a shape of an $\LIBSL$ of cardinality $n$. 

Now consider  $s,s^{\prime}\in\pi(l(n))$, with $s\neq s^{\prime}$ and let $\leq_{s},\leq_{s^{\prime}}$ be the  order of their presentations. The fact that $s\neq s^{\prime}$ implies that for at least two elements $2_{x}^{e_{i}},2_{x^{\prime}}^{e_{i}^{\prime}}\in P^{+}(s),P^{+}(s^{\prime})$  (with $e_{i}\neq e_{i}^{\prime}$) it holds 
$2_{x}^{e_{i}}\leq_{s}2_{x^{\prime}}^{e_{i}^{\prime}}\iff2_{x^{\prime}}^{e_{i}^{\prime}}\leq_{s^{\prime}}2_{x}^{e_{i}}.
$
 This, by applying Remark \ref{L-ibsl is  sequence}, proves that the shapes described by $P(s), P(s^{\prime})$ are different. 
Finally, consider  $\mathbb{X}_{\B}=\langle\A_{i}, I\rangle$ a shape of an $\LIBSL$ $\B\cong\PL(\mathbb{X})$ with cardinality $n$  such that  $\mid I\mid=k$ and such that $\A_{i}$ is a family of Boolean algebras whose cardinalities are exactly the members of $P(l(n))$. By Remark \ref{L-ibsl is  sequence}, $\mathbb{X}_{\B}$ can be described by the presentation $P(r)$ of an appropriate sequence $r$. Moreover, the fact that the cardinalities of the algebras in the family $\A_{i}$ are all and only the members of $P(l(n))$, implies that $r$ can be obtained from $l(n)$ by permuting at least one element $2_{x}^{e_{i}}\in P^{+}(l(n))$ with another $2_{x^{\prime}}^{e_{i}^{\prime}}\in P^{+}(l(n))$ such that $e_{i}\neq e_{i}^{\prime}$. By construction, this implies $r\in\pi(l(n))$, as desired.
\qed
\end{proof}

\begin{corollary}
Let $l_{1}(n),...,l_{p}(n)$ be the elements of $L(n)$. Then the number of shapes of \LIBSL of order $n$ is equal to 
$\sum_{i=1}^{p}c(l_{i}(n)).$
\end{corollary}

\begin{algorithm}
\caption{}\label{alg:genandcount}
\begin{algorithmic}[1]
\Function{L-IBSL}{$n$}
	\State $F=$\Call{GenForest}{$n$} \label{lin:tree}
	\ForEach{tree $T$ in $F$} 
		\ForEach{branch $B$ in $T$}
			\Comment{Notice that $B$ is a list of couples $(i,2^e)$}
			\State $B'=B$ without all the couples $(i,2^0)$
			\ForEach{permutation $P$ of $B'$}\label{lin:perm}
				\ForEach{ $(i,2^{e_i})\to(j,2^{e_j})$ in $P$}
					\State $t=t\times$\Call{N}{$2^{e_i}$, $2^{e_j}$}\label{lin:napp}
				\EndFor
			\EndFor
		\EndFor
	\EndFor
	\State \textbf{return} $t$
\EndFunction
\end{algorithmic}
\end{algorithm}

\begin{theorem}
The number of all the non-isomorphic \LIBSL of cardinality $n$ is given by \textsc{L-IBSL}$(n)$.
\end{theorem}
\begin{proof}
Let $B$ be one of the branches of a tree in $F$, as computed in Line~\ref{lin:tree}.
By Theorem~\ref{th: algoritmo completo}, $B$ is exactly a sequence $l(n)\in L(n)$.
Each one of the $c(l(n))$ permutations (see Theorem~\ref{teorema principale}) of $l(n)$ are computed 
in Line~\ref{lin:perm}. Consider $s\in \pi (l(n))$. 
In the light of Lemma \ref{lem: conteggio hom con isofilter}, for every pair $2^{e_{i}},2^{e_{j}}$ such that $2^{e_{i}}\to2^{e_{j}}\in s$ the number
$\prod_{i,j}N(2^{e_{i}},2^{e_{j}})$
gives us all the non-isomorphic \LIBSL of the shape described by  $s$, as computed in Line~\ref{lin:napp}.
\qed
\end{proof}

\section{Conclusions}\label{sec: Conclusioni}

The following results have been obtained on a GNU/Linux Debian 4.9.82-1 system with an Intel Core i7-5500U CPU and 8GB of RAM
\footnote{
The Python implementation, the Mace4 input and output files can be downloaded from:
\url{https://homes.di.unimi.it/~valota/code/libsl.zip}}.

To study the effectiveness of our algorithm,
we have used Mace4 \cite{prover9-mace4} to compute the number of finite linearly ordered IBSLs, 
relying on the first-order theory provided in Section~\ref{sez: classe l-ordered}.
We point out that using this FO theory, 
in the LIBSL given by the P\l onka sum of $n$ trivial Boolean algebras we have $0=1$. 
Mace4 assumes $0\not=1$, and hence such type of LIBSLs are not generated by it. 
To obtain P\l onka sums of trivial Boolean algebras, we have to replace $1$ for another constants
in the FO theory.

Mace4 produces in reasonable time the algebraic structures of cardinality up to $11$.
For cardinality $12$, Mace4 exits reaching its internal time limit.
After running the \texttt{isofilter} program associated to Mace4, we obtain a file containing
the non-isomorphic LIBLs of cardinality $2\leq n\leq 11$ generated by Mace4. 

The procedure introduced in the previous section has been implemented in Python and
has been used to compute the number of all the non-isomorphic LIBSLs of cardinality $1\leq n\leq 23$,
the results are reported in Table~\ref{tab:fine}. 
For cardinality $n=24$, the script uses too much RAM and was automatically killed by the system.
To measure the running times of both experiments we have used the Debian GNU/Linux 
command-line tool \texttt{time}, the results are summarized in the following table.
\begin{table}[h]
\begin{center}						
\begin{tabular}{|c|c|c|c|c|}
Running Times	& Algorithm~\ref{alg:genandcount} 	& Mace4 	& \texttt{interpformat}	&\texttt{isofilter}\\	
\hline				
\texttt{real} 	&0m1.331s 				&1m1.115s	&0m40.891s			&0m43.611s\\
\hline
\texttt{user} 	&0m1.176s 				&1m0.008s	&0m40.484s			&0m43.524s	\\
\hline
\texttt{sys}	&0m0.156s 				&0m0.908s	&0m0.232s			&0m0.068s \\
\hline
\end{tabular}
\end{center}
\caption{The running times of our experiments as calculated by the tool \texttt{time}.
The second column reports the total time used by the Python implementation of our algorithm 
to count all the non-isomorphic LIBSLs of cardinality $1\leq n\leq 23$.
The third column reports the time used by Mace4 to generate the first-order models of cardinality $2\leq n\leq 11$.
The fourth column reports the time used by \texttt{interpformat} to transform the Mace4 models in a format useful for \texttt{isofilter}.
The fifth column reports the time used by \texttt{isofilter} to produce a file with all the non-isomorphic LIBSLs of cardinality $2\leq n\leq 11$.
}
\end{table}

Comparing these running times, it is clear that also a non-optimized implementation of our algorithm is more efficient 
than the brute-force approach of Mace4, for counting purposes. 
We should point out, however,
that Mace4 generates full models with tables for each operation,
whereas our algorithm only produces shapes of LIBSLs and then it performs counting computations.
As reported above, our Python script is memory consuming.
Hence, it would be interesting to improve our implementation with a better memory management,
and with additional options to generate also the algebraic structure of the counted LIBSLs.
Moreover, computational complexity study and asymptotic analysis of our algorithm seems to be within reach,
allowing us to establish upper and lower bounds on the number of LIBSLs of cardinality n. 
These research directions are outside the scope of this paper, and are
left as future work.

\begin{table}[h]
\begin{center}						
\begin{tabular}{|c|c|c|c|c|c|c|c|c|c|c|c|c|c|c|c|c|c|c|c|c|c|c|c|}
\hline								%
$n$			&1& 2&3&4&5& 6 & 7 & 8 & 9 & 10 & 11 & 12 & 13 & 14 & 15 & 16 & 17 & 18 & 19 & 20&21 & 22 & 23  \\
\hline
\textsc{L-IBSL}($n$) 	&1&2&2&4&4 & 7& 7 & 14 & 14 & 26& 26 & 52 & 52 & 99 & 99 & 199 & 199 & 386 & 386 & 772 & 772& 1508& 1508 \\
\hline
\end{tabular}
\end{center}
\caption{The numbers of L-IBSL with $n$ elements for $1\leq n\leq 23$.}\label{tab:fine}
\end{table}
The set of cardinalities of the finite members of a variety of algebras $\mathcal{V}$,
goes by the name of \emph{fine spectrum} of $\mathcal{V}$ and it has been introduced by Taylor in \cite{T75}.
According to Quackenbush, when dealing with ordered structures,
``the fine spectrum problem is usually hopeless''  \cite{Q82}.
In this note we have introduced a procedure to count a specific subclass, namely linearly ordered IBSL: this represents a first step to solve the fine spectrum problem for the variety $\mathcal{IBSL}$. As our approach relies on the algebraic representation theorem, it can possibly be extended to different subclasses of $\mathcal{IBSL}$. An option is considering, for instance, the quasi-variety of $\mathcal{IBSL}$ whose maps, in the P\l onka representation, are injective (this class appears to be useful in the study of probability measures \cite{BonzioFlaminio}).
Since the index set of the P\l onka sum representation of a finite IBSL
is a finite semilattice, and finite semilattices coincides with finite lattices,
our next step is to improve our approach with the algorithm to generates finite lattices
established in \cite{HJ02}.
Finally, we notice that our approach is heavily grounded on the duality of Theorem~\ref{th:finibsldual}.
Hence, it appears that spectra problems are easier to handle when restated in dual terms.
Indeed, in literature one can find several duality-based solutions to the \emph{free spectrum problem} 
(counting the number of $k$-generated free algebras) for varieties related to many-valued logics\footnote{For a complete account on Lindenbaum algebras representations in many-valued logic setting, we refer the
interested reader to \cite{AgBoGe11}} \cite{ABM07,BV11},
and very recently this dual approach has been used to compute the fine spectrum of the variety
of prelinear Heyting algebras \cite{V18}.

\paragraph*{Acknowledgments:} the authors wish to thank the anonymous referees for their helpful comments.


\end{document}